\documentclass[a4paper,11pt]{article}

\usepackage{hyperref}

\usepackage{amsmath}
\usepackage{amsthm}
\usepackage{amsfonts}
\usepackage{amssymb}

\usepackage{verbatim}
\usepackage[bottom,symbol]{footmisc}
\usepackage{enumerate}
\usepackage[numbers]{natbib}
\theoremstyle{plain}
\newtheorem{theorem}{Theorem}
\newtheorem{lemma}[theorem]{Lemma}
\newtheorem{corollary}[theorem]{Corollary}

\theoremstyle{definition}
\newtheorem{definition}[theorem]{Definition}
\newtheorem{conjecture}[theorem]{Conjecture}
\newtheorem{question}[theorem]{Question}
\newtheorem{problem}[theorem]{Problem}

\theoremstyle{remark}
\newtheorem{remark}{[theorem]Remark}
\newtheorem{example}[theorem]{Example}

\title{Sperner's problem for $G$-independent families}

\author{Victor Falgas-Ravry\thanks{Institutionen f\"or matematik och matematisk statistik, Ume{\aa}  Universitet, 901 87 Ume{\aa}, Sweden. Supported by a postdoctoral grant from the Kempe foundation. Email: {\tt victor.falgas-ravry@math.umu.se}}}

\begin{document}
\maketitle
\begin{abstract}
Given a graph $G$, let $Q(G)$ denote the collection of all independent (edge-free) sets of vertices in $G$. We consider the problem of determining the size of a largest antichain in $Q(G)$. When $G$ is the edge-less graph, this problem is resolved by Sperner's Theorem. In this paper, we focus on the case where $G$ is the path of length $n-1$, proving the size of a maximal antichain is of the same order as the size of a largest layer of $Q(G)$. 
\end{abstract}

\section{Introduction}

\subsection{The $G$-independent hypercube: definition and motivation}
Let $n \in \mathbb{N}$ and let $G=(V,E)$ be a graph on $V(G)=[n]=\{1,2\ldots \ n\}$. 
\begin{definition}
A subset $A \subseteq [n]$ is \emph{$G$-independent} if $A$ is an edge-free set of vertices in $G$. The \emph{$G$-independent hypercube} $Q(G)$ is the collection of all $G$-independent subsets of $[n]$.
\end{definition}
$G$-independent hypercubes are our main object of study in this paper. By definition, the $G$-independent hypercube is a subset of the $n$-dimensional hypercube $Q_n$. Indeed, if $G$ is the graph with no edges then $Q(G)$ is exactly $Q_n$, the collection of all subsets of $[n]$.

We will be particularly interested in $Q(G)$ when $G$ is the path of length $n-1$, $P_n$, or the cycle of length $n$, $C_n$. These can be thought of as the collection of zero-one strings of length $n$ with no consecutive ones (with winding round in the case of $C_n$). These are natural combinatorial spaces, which have already appeared in a variety of contexts. Considered as graphs, the $G$-independent hypercubes $Q(P_n)$ and $Q(C_n)$ have been studied as an efficient network topology in parallel computing~\cite{Hsu93, HsuChungDas97, Stojmenovic98}. In this setting, they are known as the \emph{Fibonacci cube} and the \emph{Lucas cube} respectively.  Cohen, Fachini and K\"orner~\cite{CohenFachiniKorner10} gave bounds for the size of large antichains in $Q(P_n)$ in connection with skewincidence, a new class of problems lying halfway between intersection problems and capacity problems for graphs. Talbot~\cite{Talbot03} proved a direct analogue of the Erd\H os--Ko--Rado theorem~\cite{ErdosKoRado61} for the Lucas cube $Q(C_n)$. To state his result, we need to make a standard definition. 
\begin{definition}
Let $r$ be an integer with $0 \leq r \leq n$. The \emph{$r^{\textrm{th}}$ layer} of the $G$-independent hypercube, denoted by $Q^{(r)}(G)$, is the collection of all $G$-independent subsets of $[n]$ of size $r$.
\end{definition}
We can now state Talbot's theorem.
\begin{theorem}[Talbot]
Let $\mathcal{A} \subseteq Q^{(r)}(C_n)$ be a family of pairwise intersecting sets, and let $\mathcal{A}^{\star}$ be the collection of all $C_n$-independent $r$-sets containing $1$. Then $\vert \mathcal{A} \vert \leq \vert \mathcal{A}^{\star} \vert$.
\end{theorem}
Talbot's proof used an ingenious cyclic compression argument and easily adapts to the $Q(P_n)$ setting as well. In this case, the study of $Q(C_n)$ was motivated by a conjecture of Holroyd and Johnson~\cite{Holroyd99} on the independence number of a vertex-critical subset of the Kneser graph first identified by Schrijver~\cite{Schrijver78}.

\subsection{Antichains and $G$-independent families}\label{section: antichain in G-indep}
Our efforts in this paper are directed towards finding $G$-independent analogues of another classical combinatorial result in the hypercube, namely Sperner's theorem.
\begin{definition}
A subset of the hypercube $\mathcal{A} \subseteq Q_n$ is an \emph{antichain} 
if for all 
$A,B \in \mathcal{A}$ with $A \neq B$, $A$ is not a subset of $B$ and $B$ is not a subset of $A$. 
\end{definition}
How large an antichain can we find? Clearly for all integers $r$ with $0\leq r \leq n$, the $r^{\textrm{th}}$ layer of $Q_n$ is an antichain. So certainly we can find an antichain at least as large as the largest layer of $Q_n$, and a celebrated theorem of Sperner~\cite{Sperner28} asserts this is in fact the best we can do.
\begin{theorem}[Sperner's Theorem]
Let $n \in \mathbb{N}$, and $\mathcal{A} \subseteq Q_n$ be an antichain. Then
\begin{align*}
\vert \mathcal{A} \vert &\leq \max_r \vert Q_n^{(r)} \vert = \binom{n}{\lfloor n/2 \rfloor}.
\end{align*}
\end{theorem}

We consider the following generalisation of Sperner's problem.
\begin{problem}\label{general sperner problem for Gindep}
Let $n \in \mathbb{N}$, and let $G$ be a graph on $[n]$. What is the maximum size of an antichain in $Q(G)$?
\end{problem}
Write $s(G)$ for the maximum size of an antichain in $Q(G)$. We call $s(G)$ the \emph{width} of $Q(G)$. As in Sperner's theorem the size of a largest layer in $Q(G)$ gives us a lower bound on the width $s(G)$. This is not sharp in general: if $G$ is the star on $[n]$ with edges $\{1i:\ 2 \leq i \leq n\}$, then it is easy to see that $s(G)$ is larger than the largest layer of $Q(G)$ by $1$. The width $s(G)$ can in fact be much larger than a largest layer of $Q(G)$, as the following example shows.
\begin{example}
Let $m\in \mathbb{N}$. Let $G$ be a complete multipartite graph having for each integer $i\in [m]$ exactly $\lfloor 2^{2^m-2^i+\frac{i}{2}}\rfloor$ parts of size $2^i$.
\end{example}
The graph $G$ in the example above is $T$-partite, where $T=\sum_{i=1}^m \lfloor 2^{2^m-2^i+\frac{i}{2}}\rfloor$, and has $n=\sum_{i=1}^m \lfloor 2^{2^m-2^i+\frac{i}{2}}\rfloor 2^i$ vertices, which is of order $2^{2^m}$. A set of vertices in $G$ is independent if and only if it meets at most one of the parts of $G$. An antichain in $Q(G)$ is therefore the disjoint union of a collection of antichains, each lying inside a distinct part of $G$. It then follows from Sperner's theorem that the size of a maximal antichain in $Q(G)$ is
\begin{align*}
s(G)&=\sum_{i=1}^m \lfloor 2^{2^m-2^i+\frac{i}{2}}\rfloor \binom{2^i}{2^{i-1}}=\sum_{i=1}^m 2^{2^m-2^i+\frac{i}{2}} \frac{2^{2^i}}{\sqrt{2^i}}\sqrt{\frac{2}{\pi}}\left(1+O\left(\frac{1}{i}\right)\right)\\
&=\sqrt{\frac{2}{\pi}} m 2^{2^m}(1+o(1)).
\end{align*}
(Here in the first line we have used Stirling's approximation for the factorial.)

On the other hand, the layers of $Q(G)$ are much smaller: the size of the $r^{\textrm{th}}$ layer oscillates between peaks which have order $2^{2^m}$, one for each $i$ with $1 \leq i \leq m$. These peaks occur when $r$ is close to $2^{i-1}$, and correspond to the largest layer for the parts of size $2^i$. Close to the peak corresponding to $i$, the sum of the contribution from the parts of size $2^j$ for $j \neq i$ has order dominated by the contribution from the parts of size $2^i$. It follows that 
\begin{align*}
\max_{0\leq r \leq n} \vert Q^{(r)}(G)\vert = O\left(2^{2^m}\right)=o\left(s(G)\right).
\end{align*}
In general $s(G)$ and $\max\{ \vert Q^{(r)}(G)\vert: \ 0 \leq r \leq n\}$ need thus not even be of the same order.
\begin{question}\label{eqinsperner}
When is 
\[s(G)= \max_{0\leq r \leq n} \vert Q^{(r)}(G)\vert?\]
\end{question}

A natural guess is that it is sufficient for most vertices in $G$ to look more or less the same. Let $G$ be a graph. Recall that an \emph{automorphism} of $G$ is a bijection $\phi:\ V(G) \rightarrow V(G)$
such that $\phi$ maps edges to edges and non-edges to non-edges. A graph is \emph{vertex transitive} if for every $x,y \in V(G)$ there exists an automorphism of $G$ mapping $x$ to $y$.
\begin{conjecture}\label{transitiveconjecture}
Let $G$ be a vertex-transitive graph. Then 
\[s(G)= \max_{0\leq r\leq n}\vert Q^{(r)}(G)\vert.\]
\end{conjecture}
Of course, vertex-transitivity is not a necessary condition for the width of $Q(G)$ to coincide with the size of the largest layer. Indeed, consider the complete graph on $n$ vertices with one edge removed. This is not vertex-transitive, but the largest antichain is exactly the largest layer, i.e. the collection of all singletons.  Similarly the path $P_n$, while not vertex-transitive, is close to the vertex-transitive cycle $C_n$, and we believe the conclusion of Conjecture~\ref{transitiveconjecture} holds for $G=P_n$ also.
\begin{conjecture}\label{fibonacciconjecture}
\[s(P_n)= \max_{0\leq r \leq n}\vert Q^{(r)}(P_n)\vert.\]
\end{conjecture}

\subsection{Results and structure of the paper}
In their study of skewincident families, Cohen, Fachini and K\"orner~\cite{CohenFachiniKorner10} found themselves needing to give a bound on $s(P_n)$. They showed
\begin{align*}
s(P_n) \leq \vert Q(P_{n-1}) \vert= \left(\frac{2}{1+\sqrt{5}} +o(1)\right) \vert Q(P_n) \vert,
\end{align*}
a bound which was sufficient for their purposes, but which, as they observed, is fairly weak. They asked for the value of $s(P_n)$, and remarked that none of the classical proofs of Sperner's theorem seemed to adapt to this setting. The main purpose of this paper is to try and answer their question. We shall focus on $Q(P_n)$ and Conjecture~\ref{fibonacciconjecture}, though our techniques also apply in a more general setting (see Theorem~\ref{general G theorem} in Section~\ref{section: general G}). 
We show the following.
\begin{theorem}\label{fibonaccitheorem}
There exists a constant $C>1$ such that
\[s(P_n) \leq C \max_{0\leq r \leq n} \vert Q^{(r)}(P_n) \vert.\]
\end{theorem}
This improves the earlier bound of Cohen, Fachini and K\"orner~\cite{CohenFachiniKorner10} by a multiplicative factor of $O(n^{-1/2})$. It is however a far cry from Conjecture~\ref{fibonacciconjecture}, and in addition has a rather calculation-intensive 
proof.

Our paper is structured as follows. In Section~\ref{section: preliminaries}, we run through some  preliminaries. In Section~\ref{section: main theorem}, we prove Theorem~\ref{fibonaccitheorem}. We then prove small cases of Conjecture~\ref{fibonacciconjecture} in Section~\ref{section: small cases}, and briefly discuss why some classical proofs of Sperner's theorem do not adapt well to the Fibonacci cube setting. In Section~\ref{section: general G} we explain how the proof of Theorem~\ref{fibonaccitheorem} can be made to work in a more general setting. We end in Section~\ref{section: conclusion} with some questions on isoperimetric problems in $Q(P_n)$.

\section{Preliminaries}\label{section: preliminaries}

\subsection{Counting in the Fibonacci cube}
The \emph{Fibonacci sequence} $(F_n)_{n \in \mathbb{Z}_{\geq 0}}$ is the sequence defined by the initial values $F_0=0$, $F_1=1$ and the recurrence relation $F_{n+2}=F_{n+1}+F_n$ for $n \geq 0$. It is a well-known fact (and an easy exercise) that the sizes of Fibonacci cubes are given by terms of the Fibonacci sequence: $\vert Q(P_n)\vert= F_{n+2}$.
We now compute the size $q_n^r= \vert Q^{(r)}(P_n)\vert$ of a layer in $Q(P_n)$. 
\begin{lemma}\label{layer size}
$q_n^r= \binom{n-r+1}{r}$.
\end{lemma}
(We follow the standard convention that a binomial coefficient $\binom{a}{b}$ with $b>a$ or $b<0$ evaluates to zero.)  
\begin{proof} This is again an easy exercise in enumeration, but as we use the same counting technique later on in the paper, we write out the proof in full here.

Note that $Q^{(r)}(P_n)$ is empty for $r > \lceil n/2 \rceil$, so we may assume $r \leq \lceil n/2 \rceil$. We build all zero-one sequences of length $n$ containing exactly $r$ ones and such that all ones are separated by at least one zero as follows. We begin with the separated sequence $1010101 \ldots 01$ of length $2r-1$ and containing $r$ ones and $(r-1)$ zeroes. Then we insert zeroes in the $(r+1)$ `bins' defined by the gaps between successive $1$s, the gap to the left of the leftmost $1$ and the gap to the right of the rightmost $1$. We have $n-2r+1$ zeroes to insert into these bins. The number of ways of partitioning $n-2r+1$ objects into $r+1$ labelled lots is just $\binom{n-r+1}{r}$, proving our claim.
\end{proof}
Next, let us identify the largest layers of $Q(P_n)$. 
\begin{lemma}\label{max layer}
Let $r_{\star}$ be an integer maximising the layer size $\vert Q^{(r)}(P_n)\vert$. Then,
\[r_{\star}=\left\lceil \frac{5n+2 - \sqrt{ 5n^2+20n+24}}{10}\right\rceil\] 
 or 
\[r_{\star}=\frac{5n+2 - \sqrt{ 5n^2+20n+24}}{10}+1. \]
\end{lemma}
\begin{remark}
The maximal layer thus satisfies $r_{\star}= \frac{5-\sqrt{5}}{10}n+O(1)$, and is unique unless $5n+2 -\sqrt{5n^2+20n+24}$ is an integer multiple of $10$.
\end{remark}
\begin{proof}
We consider the ratio between the sizes of two consecutive layers of $Q(P_n)$.
\begin{align*}
\frac{\vert Q^{(r+1)}(P_n)\vert}{\vert Q^{(r)}(P_n)\vert} &= \binom{n-r}{r+1}/\binom{n-r+1}{r}
\end{align*}
This is greater or equal to $1$ if and only if $r$ satisfies
\[5r^2-r(5n+2)+(n^2-1)\geq 0,\]
which in the range $0 \leq r \leq \lceil n /2 \rceil $ happens if and only if 
\[r \leq \frac{1}{10}\left\{5n+2 - \sqrt{ 5n^2+20n+24}\right\}.\] 
The lemma follows.
\end{proof}

Now let us consider $Q(P_n)$ as a directed graph $D(P_n)$ by setting a directed edge from $A$ to $B$ if $B=A \cup \{b\}$ for some $b \notin A$, i.e. if $B$ covers $A$ in the partial order induced by $\subseteq$.
\begin{definition}
The \emph{in-degree} $d^-(A)$ of a set $A \in Q(P_n)$ is the number of edges of $D(P_n)$ directed into $A$, while the \emph{out-degree} $d^+(A)$ is the number of edges of $D(P_n)$ directed out of $A$.
\end{definition}
Given a set $A \in Q^{(r)}(P_n)$, its in-degree $d^-(A)$ is always exactly $r$; however, as we shall see next, its out-degree could take any integer value between $n-3r$ and $n-2r$.

Write $Q^{(r, d)}(P_n)$ for the collection of elements of $Q^{(r)}(P_n)$ with out-degree equal to $d$, and let $q^{r,d}_n= \vert Q^{(r,d)}(P_n)\vert$.
\begin{lemma}\label{dexact}
$q^{r,d}_n = \binom{r+1}{d-n+3r} \binom{n-2r}{n-2r-d}$.
\end{lemma}
\begin{proof}
We can characterise the out-degree in terms of `empty bins'. Recall that in Lemma~\ref{layer size} we built $Q^{(r)}(P_n)$ from the zero-one sequence of length $2r-1$, $1010\ldots 101$ by placing the $n-2r+1$ remaining zeroes into the $r+1$ `bins' defined by the gaps between consecutive $1$s. Suppose $i$ zeroes have been placed in bin $j$. Then the corresponding interval of zeroes will contribute $i-1$ to the out-degree. Thus the out-degree associated with a zero-one sequence $\mathbf{s}$ is 
\begin{align*}
d&=n-2r+1 -(r+1-z(\mathbf{s}))=n-3r+z(\mathbf{s}),
\end{align*}
 where $z(\mathbf{s})$ is the number of bins which have not received any zero.

Now, how many of our zero-one strings have $z$ empty bins? There are $\binom{r+1}{z}$ ways of choosing the bins which will be empty, whereupon we need to put at least one zero into the remaining $r+1-z$ bins. We then have to allocate the remaining $n-2r+1-(r+1-z)= n-3r+z$ zeroes to the $r+1-z$ non-empty bins; there are, as we observed in the proof of Lemma~\ref{layer size}, $\binom{n-2r}{r-z}$ ways of doing this. Setting $z=d-n+3r$ concludes the proof of the lemma.
\end{proof}
Note Lemma~\ref{dexact} implies that $q_n^{r,d}\neq 0$ if and only if $n-3r\leq d \leq n-2r$. These bounds are attained by, for example, the zero-one sequence consisting of $r$ $010$-blocks followed by a single block consisting of $n-3r$ zeroes (out-degree $n-3r$), and the zero-one sequence consisting of $r$ $10$-blocks followed by a single block consisting of $n-2r$ zeroes (out-degree $n-2r$). These two examples are the extremes we have to contend with inside a layer of the Fibonacci cube.

Lemma~\ref{dexact} has the following corollary.
\begin{corollary}\label{max outdegree}
Let $r,n$ be fixed, and let $d_{\star}=d_{\star}(r,n)$ be an integer maximising $q^{r,d}_n$. Then
\[d_{\star} = \left\lceil \frac{(n-2r)^2+2n-5r-1}{n-r+3}\right\rceil\]
or
\[ d_{\star} = \frac{(n-2r)^2+2n-5r-1}{n-r+3}+1.\]
\end{corollary}
Thus if $r= \alpha n$ for some $\alpha>0$, then the most common out-degree in $Q^{(r)}(P_n)$ is $d_{\star}(r,n)= \frac{(1-2\alpha)^2}{1-\alpha}n +O(1)$. Before we give a proof of Corollary~\ref{max outdegree}, let us give a heuristic justification of why we expect $d_{\star}$ to be about this. In the proof of Lemma~\ref{dexact} we established a correspondence between out-degree and (roughly speaking) the number of occurences of gaps of length one between successive $1$s (ie  occurences of $101$). Now what is the probability that the gap between the first two $1$s has length $1$? Contracting a gap of length $1$ between the first two $1$s gives us a member of $Q^{(r-1)}(P_{n-2})$. Thus the likelihood of this occuring is roughly 
\[\vert Q^{(r-1)}(P_{n-2})\vert /\vert Q^{(r)}(P_{n})\vert=r/(n-r+1)\approx \alpha/(1-\alpha)\]
when $r=\alpha n$. Since there are $r+1 \approx \alpha n$ gaps, the expected number of short gaps is $z\approx n\alpha^2/(1-\alpha)$, which implies in turn that the expected out-degree is $d=n-3r+z\approx n(1-2\alpha)^2/(1-\alpha)$. Unsurprisingly the maximum of $q^{r,d}_n$ is attained when $d$ is close to the expected out-degree. Having said this, we turn to a formal argument.

\begin{proof}[Proof of Corollary~\ref{max outdegree}]
Consider the ratio $q^{r,d+1}_n/q^{r,d}_n$. By Lemma~\ref{dexact}, this is equal to 
\begin{align*}
\frac{q^{r,d+1}_n}{q^{r,d}_n}&= \binom{r+1}{d+1-n+3r} \binom{n-2r}{n-2r-d-1}/ \binom{r+1}{d-n+3r}\binom{n-2r}{n-2r-d}. 
\end{align*}
Solving the associated linear inequality, we see that $q^{r,d+1}_{n}/q^{r,d}_n\leq 1$ if and only if 
\begin{align}
d\leq \frac{(n-2r)^2+2n-5r-1}{n-r+3}, \label{ineq: degree}
\end{align}
with equality if and only if we have equality in~(\ref{ineq: degree}). The Corollary follows.
\end{proof}
\begin{remark}\label{monotonicity degree}
Note that the proof of Corollary~\ref{max outdegree} establishes in fact that $q^{r,d}_n$ is strictly increasing in $d$ until it hits its (at most two) maxima, and then becomes strictly decreasing in $d$. We shall use this monotonicity later on.
\end{remark} 
\begin{corollary}\label{value dstar} Let $r_{\star}$ be an integer maximising $q_n^r$, and let $r=r_{\star}+c\sqrt{n}$ for some $c \in[-\sqrt{\log n}, + \sqrt{\log n}]$. Then for $d_{\star}(r,n)$ an integer maximising $q^{r,d}$, we have
\[d_{\star}(r,n)= \left(\frac{5-\sqrt{5}}{10}\right)n - \left(\frac{5\sqrt{5}-7}{2}\right) c\sqrt{n}+(20-8\sqrt{5})c^2 +O(1).\]
\end{corollary}
\begin{proof}
This is a straightforward calculation from Corollary~\ref{max outdegree}, from the fact $r_{\star}= \frac{5-\sqrt{5}}{10}n +O(1)$ (Lemma~\ref{max layer}), and from the hypothesis on $r$:
\begin{align*}
d_{\star}(r,n)&= \frac{(n-2r)^2}{n-r}+O(1)= \frac{\left(\frac{2\sqrt{5}}{10}n-2c\sqrt{n}\right)^2}{\frac{5+\sqrt{5}}{10}n-c\sqrt{n}}+O(1),
\end{align*}
which, expanded to second order, yields the desired result.
\end{proof}

\subsection{Concentration}
With the combinatorial preliminaries out of the way, let us obtain some concentration results for $q^r_n=\vert Q^{(r)}(P_n)\vert $ and $q^{r,d}_n=\vert Q^{(r,d)}(P_n)\vert $. Given the binomial coefficients appearing in Lemmas~\ref{layer size} and~\ref{dexact}, we expect Chernoff-type concentration of both the weight in $Q(P_n)$ around the heaviest layer(s) $Q^{(r_{\star})}(P_n)$ and of the out-degrees in $Q^{(r)}(P_n)$ around the likeliest out-degree(s) $d_{\star}=d_{\star}(r,n)$. By double counting, we also expect, analogously to $Q_n$, that the largest layer in $Q(P_n)$ will occur when the in-degree and the average out-degree are the same -- that is, by the observation after Corollary~\ref{max outdegree}, when $r \approx (n-2r)^2/(n-r)$. Solving this yields $r \approx \frac{(5-\sqrt{5})}{10}n$, matching the estimate we made after Lemma~\ref{max layer} and giving perhaps better intuition as to why the maximum occurs at this point.

These heuristic observations we have made regarding concentration are indeed correct, and can be proved formally using Stirling's approximation, 
\[m!=\left(1+O\left(\frac{1}{m}\right)\right)\sqrt{2\pi m} \left(\frac{m}{e}\right)^m,\]
and some simple calculus.

Let $F$ be the function 
\[F: \ x\mapsto (1-x) \log (1-x)-x \log x -(1-2x) \log (1-2x).\]
\begin{lemma}\label{layer size approx}
Let $\alpha=\alpha(n)$ be a sequence of real numbers with ${10}^{-9}< \alpha(n) < \frac{1}{2}-{10}^{-9}$ and $\alpha n \in \mathbb{N}$ for $n\geq 4$. Then 
\[  q^{\alpha n}_n= 
\left(\frac{(1-\alpha)\sqrt{1-\alpha}}{\sqrt{2\pi \alpha (1-2\alpha)}(1-2\alpha)} +O\left(\frac{1}{n}\right)\right) n^{-1/2} \exp \left(n F(\alpha)\right).\]

\end{lemma}
\begin{proof}
This is a straightforward calculation from Lemma~\ref{layer size} and Stirling's formula:
\begin{align*}
q^{\alpha n}_n&=\binom{n-\alpha n +1}{\alpha n}
= \frac{ ((1-\alpha)n)!} {(\alpha n)! ((1-2\alpha)n)!} \left(\frac{(1-\alpha)n+1} {(1-2\alpha)n+1}\right).
\end{align*}
Substituting Stirling's approximation in the above (which we can do since $\alpha$ and $1-2\alpha$ are both bounded away from $0$) then yields the claimed equality. 
\end{proof}
As expected given that the maximum of $q^{\alpha n}_n$ occurs when $\alpha = \frac{5-\sqrt{5}}{10} + O(n^{-1})$, we find that $F$ attains a global maximum at $\frac{5-\sqrt{5}}{10}$:
\[ F'(x)= \log \left(\frac{(1-2x)^2}{x(1-x)}\right),\]
which is strictly positive for $x< \frac{5-\sqrt{5}}{10}$, vanishes at $\frac{5-\sqrt{5}}{10}$ and becomes strictly negative for $x> \frac{5-\sqrt{5}}{10}$. Computing the second derivative, we find $F''(\frac{5-\sqrt{5}}{10})=-5\sqrt{5}$. 
\begin{corollary}\label{layer size decay}
Let $Q^{(r_{\star})}(P_n)$ be a largest layer of $Q(P_n)$. Then the following hold:
 \begin{enumerate}[(i)]
\item if $r=r_{\star}+ c\sqrt{n}$ for some $c\in [-\sqrt{\log n},\sqrt{\log n}]$, then 
\[q^r_n= \exp\left( -\frac{5\sqrt{5} c^2}{2}+o(1)\right) q^{r_{\star}}_n;\]
\item there are 
\[O\left(n \exp\left(-\frac{5\sqrt{5}}{2}\log n\right)q^{r_{\star}}_n\right)=o(q^{r_{\star}}_n)\] 
sets in $Q(P_n)$ with size differing from $r_{\star}$ by more than $\sqrt{n\log n}$ .
\end{enumerate}
\end{corollary} 
\begin{proof}
Immediate from Lemma~\ref{layer size approx} and the calculation above.
\end{proof}

We now turn to out-degree concentration. Define 
\begin{align*}
G(x,y)&=x\log x +(1-2x)\log (1-2x) -y\log y-2(x-y) \log (x-y)\\
&  -(1-3x+y) \log (1-3x+y).
\end{align*}
\begin{lemma} \label{outdegree approx}
Let $\alpha=\alpha(n)$, and $\beta=\beta(n)$ be sequences of real numbers satisfying ${10}^{-9}< \beta(n) < \alpha(n)-{10}^{-9}$, and $ \alpha(n) <(1+\beta-{10}^{-9})/3$ and $n\alpha, n\beta \in \mathbb{N}$ for $n \geq 9$. Then,
\[q^{\alpha n, (1-3\alpha + \beta)n}_n 
=\left(\frac{\alpha\sqrt{\alpha (1-2\alpha)}}{2\pi(\alpha-\beta)^2\sqrt{\beta(1-3\alpha+\beta)}}+O\left(\frac{1}{n}\right)\right)n^{-1}\exp\left(nG(\alpha, \beta)\right).\]
\end{lemma}
\begin{proof}
This is a straightforward calculation from Lemma~\ref{dexact} and Stirling's formula:
\begin{align*}
q^{\alpha n, (1-3\alpha + \beta)n}_n&=\binom{\alpha n +1}{\beta n} \binom{(1-2\alpha)n}{(\alpha-\beta)n}\\
&=\left(\frac{\alpha n+1}{(\alpha-\beta)n+1}\right) \frac{(\alpha n)!}{((\alpha-\beta)n)!(\beta n)!} \frac{((1-2\alpha)n)!}{((\alpha-\beta)n)!((1-3\alpha+\beta)n!)}
\end{align*}
Substituting Stirling's approximation in the above then yields the claimed equality. (We can do this since $\alpha$, $\beta$, $(\alpha-\beta)$,  $(1-2\alpha)$ and $(1-3\alpha+\beta)$ are all bounded away from $0$. Note that for $n\geq 9$ there exist at least two distinct integers $m_1$ and $m_2$ with $\frac{n}{4}\leq m_1 < m_2\leq \frac{n}{3}$, and hence legal choices of $\alpha(n)$ and $\beta(n)$, so that our claim is not vacuous.)
\end{proof}
Again it is no surprise that for a fixed $\alpha$, the function $G_{\alpha}:\ y \mapsto G(\alpha, y)$ attains a global maximum at $\beta= \frac{\alpha^2}{1-\alpha}$:
\[G_{\alpha}'(y)= \log \left( \frac{(\alpha-y)^2}{y(1-3\alpha+y)}\right),\]
which is strictly positive for $y< \frac{\alpha^2}{1-\alpha}$, vanishes at $\frac{\alpha^2}{1-\alpha}$ and becomes strictly negative for $y> \frac{\alpha^2}{1-\alpha}$. Computing the second derivative, we find 
\[G_{\alpha}''\left(\frac{\alpha^2}{1-\alpha}\right)= -\frac{(1-\alpha)^3}{\alpha^2 (1-2\alpha)^2}.\] 
In particular for $\alpha = \frac{5-\sqrt{5}}{10}+O\left(\sqrt{\frac{\log n}{n}}\right)$, we have $G_{\alpha}''(\frac{\alpha^2}{1-\alpha})= \frac{25+11\sqrt{5}}{2}+O\left(\sqrt{\frac{\log n}{n}}\right)$.
\begin{corollary}\label{outdegree decay}
Let $r_{\star}$ be an integer maximising $q_n^r$, and let $r = r_{\star} + O(\sqrt{n \log n})$. Let $d_{\star}=d_{\star}(r,n)$ be an integer maximising $q_n^{r,d}$. Then
\begin{enumerate}[(i)]
\item if $d=d_{\star}+c\sqrt{n}$ for some $c \in [-2\sqrt{\log n}, + 2\sqrt{\log n}]$, then 
\[q_n^{r,d}=\exp\left(-\frac{(25+11\sqrt{5})c^2}{4}+o(1) \right) q_n^{r,d_{\star}}\]
\item there are 
\[O\left(n \exp\left(-\frac{(25+11\sqrt{5})}{4}\log n \right)q^{r, d_{\star}}_n\right)=o\left(n^{-2}q^{r,d_{\star}}_n\right)\] sets in $Q^{(r)}(P_n)$ with out-degree differing from $d_{\star}$ by more than $\sqrt{n\log n}$.
\end{enumerate}
\end{corollary}

\subsection{Summation bounds}
We shall also need the following simple bounds on a sum of exponentials.
%
\begin{lemma}\label{exponential sums are nice}
Let $p(x)= a_0 +a_1 x +a_2 x^2$ be a quadratic polynomial with $a_2 >0$. Then
\[ C_1(p)\leq \sum_{i \in \mathbb{Z}} e^{-p(i)}\leq C_2(p),\]
where $C_1(p), C_2(p)$ are strictly positive constants depending only on $a_0$, $a_1$ and $a_2$.
\end{lemma}
\begin{proof}
This is an easy exercise --- just use comparison with integrals to bound the sum, and then elementary calculus to evaluate $\int e^{-p(x)}dx$.
\end{proof}
\begin{corollary}\label{qrdstarweight}
There are constants $C_1>0$ and $C_2>0$ such that if $r$ is an integer with $\vert r-r_{\star}\vert \leq \sqrt{n \log n}$ and $d_{\star}=d_{\star}(r,n)$ is an integer maximising $q_n^{r,d}$, then
\[ C_1 \frac{q_n^r}{\sqrt{n}}\leq q_{n}^{r,d_{\star}} \leq C_2 \frac{q_n^r}{\sqrt{n}}.\] 
\end{corollary}
(We could also have proved this directly by calculating the ratio $q_n^{r,d_{\star}}/q_n^r$ using Lemmas~\ref{layer size},~\ref{dexact} and Corollary~\ref{max outdegree}.)
\begin{proof}
By Corollary~\ref{outdegree decay} part (ii) we may discard sets in $Q^{(r)}(P_n)$ with out-degree differing from $d_{\star}$ by more than $\sqrt{n\log n}$. Divide the remaining sets in $Q^{(r)}(P_n)$ into \emph{out-degree intervals} of width $\sqrt{n}$:
\[I_i= \{A\in Q^{(r)}(P_n): \ d_{\star}+i \sqrt{n} \leq d^+(A) < d_{\star}+(i+1)\sqrt{n}\},\]
for $i \in\mathbb{Z} \cap [-\sqrt{\log n}, \sqrt{\log n}]$. Then we have
\begin{align*}
q_n^r&= \sum_i \vert I_i \vert  +o(q_n^r)\\ 
&\qquad \textrm{by Corollary~\ref{outdegree decay} part (ii)}\\
&\leq \sum_{i\geq 0} q_n^{r, \lceil d_{\star}+i\sqrt{n}\rceil}\sqrt{n} + \sum_{i< 0} q_n^{r, \lfloor d_{\star}+(i+1)\sqrt{n}\rfloor}\sqrt{n}\\
& \qquad \textrm{since $q^{r,d}_n$ monotonically decreases as $d$ moves away from $d_{\star}$}\\
&= 2 \sqrt{n} \sum_{i\geq 0}\exp\left(-\left(\frac{25+11\sqrt{5}}{4}\right)i^2+o(1)\right)q_n^{r,d_{\star}}\\
&\qquad \textrm{by Corollary~\ref{outdegree decay} part (i)},
\end{align*}
which by Lemma~\ref{exponential sums are nice} is at most $\frac{\sqrt{n}}{C_1} q_{n}^{r,d_{\star}}$ for some absolute constant $C_1>0$. The inequality in the other direction follows in much the same way.
\end{proof}

\section{Proof of Theorem~\ref{fibonaccitheorem}}\label{section: main theorem}
We can now proceed to the proof of Theorem~\ref{fibonaccitheorem} proper. Let $Q^{(r_{\star})}(P_n)$ be a largest layer of $Q(P_n)$, and for every $r$ let $d_{\star}(r,n)$ be an integer maximising $q_n^{r,d_{\star}}$. By Corollary~\ref{layer size decay}, we can restrict our attention in a proof of Theorem~\ref{fibonaccitheorem} to layers $r$ with $\vert r-r_{\star}\vert \leq \sqrt{n\log n}$. We denote by $Q'(P_n)$ the corresponding subset of $Q(P_n)$. Note that for $n$ sufficiently large (say $n>100$) every element of $Q'(P_n)$ has nonzero out-degree in the directed graph $D(P_n)$.

\subsection{Dissection into blocks and overlapping trapeziums}
Let $c_1=1/100$. We divide $Q'(P_n)$ into (overlapping) \emph{blocks} of layers 
\[B_t=\bigcup\left\{ Q^{(r)}(P_n): \ r_{\star} + c_1t \sqrt{n} \leq r \leq r_{\star} + c_1(t+1)\sqrt{n}\right\},\]
each of which is roughly $c_1\sqrt{n}$ layers wide. (Here $t$ takes integer values in $[-\frac{\sqrt{\log n}}{c_1}, \frac{\sqrt{\log n}}{c_1}]$.)

If $t \geq 0$, we divide the top layer $Q^{(r_+)}(P_n)$ of $B_t$ into \emph{out-degree intervals} 
\[I_{s,t} = \bigcup \left\{ Q^{(r_+, d)}(P_n): \ d_{\star}(r_+, n) +s\sqrt{n} \leq d \leq  d_{\star}(r_+, n)+ (s+1)\sqrt{n}\right\},\]
each of which ranges over roughly $\sqrt{n}$ different out-degrees.

Each such interval $I_{s,t}$ defines a \emph{trapezium} 
\[T_{s,t}= \{A \in B_t: \exists A' \in I_{s,t} \textrm{ with }A \subseteq A'\}.\] 
For $n$ sufficiently large, the union of these (overlapping) trapeziums covers all of $B_t$ (since all sets in $B_t$ have positive out-degree).

If on the other hand $t<0$, we divide the bottom layer $Q^{(r_-)}(P_n)$ of $B_t$ into out-degree intervals
\[I_{s,t}=\bigcup\left\{ Q^{(r_-, d)}(P_n): \ d_{\star}(r_-, n) +s\sqrt{n} \leq d \leq d_{\star}(r_-, n)+ (s+1)\sqrt{n}\right\},\]
with again each interval defining a trapezium
\[T_{s,t}= \{A \in B_t: \exists A' \in I_{s,t} \textrm{ with } A'\subseteq A\}.\]
Taken together, the overlapping trapeziums $T_{s,t}$ cover all of $B_t$ in this case also.

\subsection{Strategy}
The heart of our proof of Theorem~\ref{fibonaccitheorem} is the following lemma.
\begin{lemma}\label{constantblock}
There is an absolute constant $C_3>1$ such that for every antichain $\mathcal{A} \subseteq Q(P_n)$ and every integer $t \in [-\frac{\sqrt{\log n}}{c_1}, \frac{\sqrt{\log n}}{c_1}]$ we have
\[\vert \mathcal{A} \cap B_t \vert \leq C_3 \max\{q^r_n: \ Q^{(r)}(P_n) \subseteq B_t\}.\]
\end{lemma}

Provided we are able to prove Lemma~\ref{constantblock}, Theorem~\ref{fibonaccitheorem} is straightforward from our concentration result on the layer size, Corollary~\ref{layer size decay}:
\begin{proof}[Proof of Theorem~\ref{fibonaccitheorem} from Lemma~\ref{constantblock}]

Let $\mathcal{A}$ be an antichain. Then,
\begin{align*}
\vert \mathcal{A}\vert &= \sum_t \vert \mathcal{A} \cap B_t\vert + o(q_n^{r_{\star}})\\
& \qquad \textrm{by Corollary~\ref{layer size decay} part (ii)}\\
& \leq \sum_t C_3 \max\{q^r_n: \ Q^{(r)}(P_n) \subseteq B_t\} +o(q_n^{r_{\star}})\\
& \qquad \textrm{by Lemma~\ref{constantblock}}\\
&=C_3 \left(\sum_{t\geq 0}q_n^{r_{\star}+\lceil c_1 t\sqrt{n}\rceil} + \sum_{t<0} q_n^{r_{\star}+ \lfloor c_1(t+1)\sqrt{n}\rfloor} \right)+o(q_n^{r_{\star}})\\
& \leq 2C_3 \left(\sum_{t\geq 0} \exp\left(-\frac{5\sqrt 5{c_1}^2}{2}t^2 +o(1) \right)q_n^{r_{\star}}\right)  +o(q_n^{r_{\star}})\\
& \qquad \textrm{by Corollary~\ref{layer size decay} part (ii)}\\ 
& \leq C q_n^{r_{\star}}\\
& \qquad \textrm{for some absolute constant $C>0$, by Lemma~\ref{exponential sums are nice}.}
\end{align*}
\end{proof}

Let us therefore turn to the proof of Lemma~\ref{constantblock}. This will be a \emph{shadow argument}.
\begin{definition}
Let $\mathcal{B}\subseteq Q(P_n)$ be a subset of the Fibonacci cube. The \emph{lower shadow} of $\mathcal{B}$ is the family
\[\partial^-(\mathcal{B}) = \{B \in Q(P_n): \ \exists b \notin B \textrm{ such that } B\cup\{b\} \in \mathcal{B}\}.\]
The \emph{upper shadow} of $\mathcal{B}$ is the family 
\[\partial^+(\mathcal{B})=\{B \in Q(P_n): \ \exists b \in B \textrm{ such that } B \setminus\{b\} \in \mathcal{B}\}.\]
\end{definition}
Recalling the directed graph $D(P_n)$ we associated with $Q(P_n)$, the lower shadow is the \emph{in-neighbourhood} of $\mathcal{B}$ in $D(P_n)$ while the upper shadow is the \emph{out-neighbourhood} of $\mathcal{B}$.

Let $t\geq 0$, and let $\mathcal{A}\subseteq B_t$ be an antichain contained in the block $B_t$. Write $\mathcal{A}^{(r)}$ for the \emph{r\textsuperscript{th} layer} of $\mathcal{A}$, $\mathcal{A}^{(r)}=\mathcal{A}\cap Q^{(r)}(P_n)$.

Let $\mathcal{A}^{(r_+)}$ be the topmost non-empty layer of $\mathcal{A} \subseteq B_t$. Since $\mathcal{A}=\mathcal{A}_0$ is an antichain, the family
\[\mathcal{A}_1= \left(\mathcal{A}\setminus \mathcal{A}^{(r_+)}\right)\cup \partial^- (\mathcal{A}^{(r_+)})\]
is also an antichain. Repeating this procedure with $\mathcal{A}_1$, then $\mathcal{A}_2$, etc, we can `push down' our family into the bottom layer of $B_t$. We will thus be done in the proof of Lemma~\ref{constantblock} if we can show we have not shrunk the size of our family by more than a constant factor in the process. (The $t<0$ case proceeds identically with upper shadows instead of lower shadows.)

To do this, we perform some careful accounting, and this is where our trapeziums (and, unfortunately, some tedious calculations) come in. Roughly speaking, the further away the out-degree lies from the layer's average out-degree, the more we could be shrinking our family when taking lower shadows. This effect is balanced out by the fact that the further we are from the average out-degree the fewer sets we have at our disposal.

\subsection{Shadows in the trapeziums}\label{section: shadows in trapezium}
In this subsection, we prove the case $t\geq 0$ of Lemma~\ref{constantblock} by taking shadows in trapeziums. We first introduce some notation.

Let $t\geq 0$. Let $r_-= r_{\star} + \lceil c_1 t \sqrt{n} \rceil$ and $r_+=r_{\star} + \lfloor c_1 (t+1) \sqrt{n} \rfloor$ be the size of sets in the bottom-most and top-most layers of $B_t$ respectively.
Given a family $\mathcal{C}\subseteq \mathcal{B}$, we let
\[\phi(\mathcal{C})=\left\{A \in Q^{(r_-)}(P_n): \  \exists A'\in \mathcal{C} \textrm{ such that }A \subseteq A' \right\}\] 
denote the collection of sets in the bottom-most layer of $B_t$ which are contained in an element of $\mathcal{C}$. In other words, $\phi(\mathcal{C})$ is obtained from $\mathcal{C}$ by repeatedly replacing the highest non-empty layer of $\mathcal{C}$ by its lower shadow until the entire family lies inside $Q^{(r_-)}(P_n)$.

\begin{proof}[Proof of case $t\geq 0$ of Lemma~\ref{constantblock}]
Let $\mathcal{A}$ be an antichain. Without loss of generality, we may assume $\mathcal{A} \subseteq B_t$. We shall show that
\[\vert \mathcal{A} \vert -\vert\phi(\mathcal{A})\vert\leq (1+C_4) q_n^{r_-} \]
for some absolute constant $C_4>0$, from which Lemma~\ref{constantblock} follows with $C_3=C_4+2$.

Let $\mathcal{A}_{s}=\mathcal{A}\cap T_{s, t}$ be the intersection of $\mathcal{A}$ with the trapezium $T_{s,t}$. By Corollary~\ref{outdegree decay} and the monotonicity of $q_n^r$, we have that 
\[\left\vert \bigcup \{\mathcal{A}_s: \ s \in \mathbb{Z} \setminus [-\sqrt{\log n}, \sqrt{\log n}]\}\right\vert\leq c_1 \sqrt{n} \cdot o\left(n^{-2}q_n^{r_-}\right)=o(q_n^{r_-}).\]
Thus for the purpose of proving Lemma~\ref{constantblock}, it is enough to consider only the sets $\mathcal{A}_s$ with $s\in [-\sqrt{\log n}, \sqrt{\log n}]$.

Observe that deleting an element from a set in $Q(P_n)$ can increase its out-degree by at most $3$. It follows that sets in $\phi(\mathcal{A}_s)$ have out-degree $d$ satisfying
\[ d_{\star}(r_+, n)+ s\sqrt{n} \leq d \leq d_{\star}(r_+, n)+ (s+1)\sqrt{n} + 3c_1 \sqrt{n}.\]
As $c_1=1/100$ it follows that $\phi(\mathcal{A}_s)$ is disjoint from $\phi(\mathcal{A}_{s+2})$ for all $s$ (since $3c_1 \sqrt{n}<\sqrt{n}$). In particular, sets in $Q^{(r_-)}(P_n)$ are contained in at most two distinct $\phi(\mathcal{A}_s)$, whence
\begin{align}
\left(\sum_{s} \vert\phi(\mathcal{A}_s) \vert\right)  - \left\vert \bigcup_s \phi(\mathcal{A}_s)\right\vert\leq \vert Q^{(r_-)}(P_n) \vert=q_n^{r_-}. \label{union size phi(A_s) is almost sum of sizes}
\end{align}
Now we shall show $\vert \phi(\mathcal{A}_s) \vert$ is not much smaller than $\vert \mathcal{A}_s\vert$. To obtain $\phi(\mathcal{A}_s)$ from $\mathcal{A}_s$, we repeatedly replace the highest non-empty layer by its lower shadow. Since $\mathcal{A}$ (and hence $\mathcal{A}_s$) is an antichain, we know that the shadow of the family's highest layer is disjoint from the rest of the family. Thus our only concern is that the family could be shrinking every time we take a lower shadow.

Observe that if $\mathcal{B}\subset Q^{(r)}(P_n)$ and the maximum out-degree in the lower shadow of $\mathcal{B}$ is $\Delta^+$, then, by counting edges from $\partial^-\mathcal{B}$ to $Q^{(r)}(P_n)$ we have:
\[ \vert \partial^- \mathcal{B} \vert \geq \frac{r}{\Delta^+} \vert \mathcal{B} \vert.\]
Going from $\mathcal{A}_s$ to $\phi(\mathcal{A}_s)$, the worst ratio we would have to contend with at any stage of the process is thus when $r=r_-=r_{\star} + c_1t\sqrt{n}+O(1)$ and $\Delta^+=d_{\star}(r_+, n)+ (s+1)\sqrt{n} + 3c_1 \sqrt{n}+O(1)$. Now by Lemma~\ref{max layer}, 
\[r_{\star}=\frac{5-\sqrt{5}}{10}n +O(1)\]
 and by Corollary~\ref{value dstar} 
\[d_{\star}(r_+, n)=\frac{5-\sqrt{5}}{10}n- \left(\frac{5\sqrt{5}-7}{2}\right)c_1(t+1)\sqrt{n}+(20-8\sqrt{5})c_1^2(t+1)^2+O(1).\]
A quick calculation then shows that the worst-case ratio is
\begin{align*}
\frac{r_-}{\Delta^+}&=\frac{{\frac{5-\sqrt{5}}{10}n+c_1t\sqrt{n}}}{\frac{5-\sqrt{5}}{10}n-\left(\frac{5\sqrt{5}-7}{2}c_1(t+1)-(s+1)-3c_1\right)\sqrt{n}+(20-8\sqrt{5}){c_1}^2(t+1)^2}+O\left(\frac{1}{n}\right)\\
&=\frac{1+ \frac{10}{5-\sqrt{5}}c_1tn^{-1/2}}{1+\frac{10}{5-\sqrt{5}}\left(s+1-\frac{5\sqrt{5}-7}{2}c_1t +\frac{13-5\sqrt{5}}{2}c_1\right)n^{-1/2}} +O\left(\frac{\log n}{n}\right)\\
&=1 -\frac{10}{5-\sqrt{5}}\left( s+1 - \frac{5\sqrt{5}-5}{2}c_1t+ \frac{13-5\sqrt{5}}{2}c_1\right)n^{-1/2} +O\left(\frac{\log n}{n}\right).
\end{align*}
(Note we used in the second line the fact that $t=O(\sqrt{\log n})$.)
Write $f_t(s)$ for the expression
\[f_t(s)=\frac{10}{5-\sqrt{5}}\left( s+1 - \frac{5\sqrt{5}-5}{2}c_1t+ \frac{13-5\sqrt{5}}{2}c_1\right).\]

If $f_t(s)<0$, then we have nothing to worry about: our family does not shrink as we take successive shadows. On the other hand if $f_t(s)\geq 0$, then we have
\begin{align}
\vert\phi(\mathcal{A}_s)\vert & \geq \left(1-f_t(s)n^{-1/2}+O\left(\frac{\log n} {n}\right)\right)^{c_1\sqrt{n}} \vert \mathcal{A}_s \vert\notag\\
&= \exp\left(-c_1 f_t(s) + O\left(\frac{\log n}{\sqrt{n}}\right)\right) \vert \mathcal{A}_s \vert.\label{As bounded by phi A_s}
\end{align}
We now give an upper bound on the size of $\phi(\mathcal{A}_s)$ (and hence, by (\ref{As bounded by phi A_s}), on $\vert \mathcal{A}_s \vert$) when $f_t(s)\geq 0$ using our concentration results. 
Write $s_0$ for the unique real solution to $f_t(s)=0$, 
\[s_0=-1 +\frac{5\sqrt{5}-5}{2}c_1t-\frac{13-5\sqrt{5}}{2}c_1,\]
Since $c_1=1/100$ and $t\geq 0$, we certainly have $s_0>-2$. By Corollary~\ref{value dstar},
\[d_{\star}(r_-,n)-d_{\star}(r_+,n)=\left(\frac{5\sqrt{5}-7}{2}\right)c_1 \sqrt{n}+O\left(\sqrt{\log n}\right).\]
(Since $t=O\left(\sqrt{\log n}\right)$). The out-degrees found in $\phi(\mathcal{A}_s)\subseteq Q^{(r_-)}(P_n)$ are thus at least 
\begin{align*}
\delta_s&=d_{\star} (r_+,n)+s\sqrt{n}\\
&=d_{\star}(r_-,n)+ s\sqrt{n} -\frac{5\sqrt{5}-7}{2}c_1\sqrt{n}+O\left(\sqrt{\log n}\right)\\
&= d_{\star}(r_-,n) +g(s)\sqrt{n} +O\left(\sqrt{\log n}\right)
\end{align*}
where $g$ denotes the linear function $s\mapsto s-\left(\frac{5\sqrt{5}-7}{2}\right)c_1$. As $s_0> -2$, as $c_1=1/100$ and as $s$ is an integer, it follows from the above that apart from at most two values of $s\geq s_0$ (namely $s=-1$ and $s=0$), the minimum out-degree in $\phi(\mathcal{A}_s)$ is greater than $d_{\star}(r_-, n)$ by a term of order $\sqrt{n}$. We can then use our concentration result and the monotonicity of $q_n^{r_-,d}$ away from $d_{\star}(r_-,n)$ to bound $\vert \phi (\mathcal{A}_s)\vert$ for $s\geq 1$: 
\begin{align}
\vert \phi(\mathcal{A}_s)\vert & \leq (\sqrt{n}+3c_1\sqrt{n}) q_n^{r_-, \delta_s}\notag\\
& \leq (3c_1+1)\sqrt{n} q_n^{r_-, d_{\star}(r_-,n)}\exp\left(-\left(\frac{25+11\sqrt{5}}{4}\right){g(s)}^2+o(1)\right)\label{apply cor outdegree decay}\\
& \leq (3c_1+1)C_2 q_n^{r_-}\exp\left(-\left(\frac{25+11\sqrt{5}}{4}\right){g(s)}^2+o(1)\right),\label{apply Cor qrdstarweight}
\end{align}
by applying Corollary~\ref{outdegree decay} in (\ref{apply cor outdegree decay}) and Corollary~\ref{qrdstarweight} in (\ref{apply Cor qrdstarweight}).

Now $f_t(s)\leq f_{0}(s)$ for all $t\geq 0$, so that we have
\begin{align}
&\sum_{s \geq s_0} \vert \mathcal{A}_s \vert  \leq \sum_{s \geq s_0} \vert \phi(\mathcal{A}_s) \vert \exp\left(c_1 f_t(s)+O\left(\frac{\log n}{\sqrt{n}}\right)\right) \notag \\
& \qquad \qquad  \qquad \qquad \textrm{(by~(\ref{As bounded by phi A_s}))}\notag\\
 &\leq (3c_1+1)C_2q_n^{r_-}\left( e^{c_1f_0(-1)+o(1)}+e^{c_1f_0(0)+o(1)} 
+ \sum_{s\geq 1} \exp\left(c_1f_0(s) -\left(\frac{25+11\sqrt{5}}{4}\right){g(s)}^2+o(1)\right)\right)\notag \\ &\qquad \qquad \qquad \qquad \textrm{(by~(\ref{apply Cor qrdstarweight}))}\notag \\
&\leq C_4 q_n^{r_-} \label{bound on sum A_s}
\end{align}
for some absolute constant $C_4>0$, by observing that ${g(s)}^2$ is quadratic in $s$ while $f_0(s)$ is only linear and applying Lemma~\ref{exponential sums are nice}.

We are then essentially done:
\begin{align*}
\vert \mathcal{A} \vert - \vert \phi(\mathcal{A}) \vert
&\leq \left(\sum_s \vert \mathcal{A}_s\vert -\vert \phi(\mathcal{A}_s)\vert\right) +\left(\sum_s \vert \phi(\mathcal{A}_s)\vert -\left\vert \bigcup_s \phi(\mathcal{A}_s)\right\vert \right)\\
&\leq \left(\sum_s \vert \mathcal{A}_s\vert -\vert \phi(\mathcal{A}_s)\vert\right) +q_n^{r_-}\qquad \qquad \textrm{(by~(\ref{union size phi(A_s) is almost sum of sizes}))}\\
&\leq \left(\sum_{s\geq s_0} \vert \mathcal{A}_s\vert\right) +q_n^{r_-}\leq  (C_4+1)q_n^{r_-} \qquad \textrm{(by (\ref{bound on sum A_s}))}
\end{align*}
from which it follows that
\begin{align*}
\vert \mathcal{A} \vert & \leq \vert \phi(\mathcal{A})\vert + (C_4+1)q_n^{r_-}\\
&\leq (C_4+2)q_n^{r_-},
\end{align*}
with $C_4+2$ a constant independent of $t$ and $n$ as required.
\end{proof}

The proof of the case $t<0$ of Lemma~\ref{constantblock} is essentially the same as the above, except that we use upper shadows instead of lower shadows (so as to push the family towards the largest layer rather than away from it). We 
conclude here the proof of Lemma~\ref{constantblock} and with it the proof of Theorem~\ref{fibonaccitheorem}.

\section{Small cases of Conjecture~\ref{fibonacciconjecture}}\label{section: small cases}
We have not tried to optimise the constant $C$ we get in our proof of Theorem~\ref{fibonaccitheorem}, as our methods will give a constant strictly greater than $1$ when we believe the correct answer should be exactly $1$. We have however established Conjecture~\ref{fibonacciconjecture} for some small values of $n$. Details follow below.

\subsection{Partition into chains}
A classical proof of Sperner's Theorem consists in partitioning $Q_n$ into symmetric \emph{chains}, each of which intersects the largest layer(s) of $Q_n$.
\begin{definition} An \emph{$l$-chain} in $Q(G)$ is a family of $l$ distinct elements of $Q(G)$, $\{A_1, \ldots A_l\}$, with $A_1 \subset A_2 \subset \ldots \subset A_l$. 
\end{definition}
 If Conjecture~\ref{fibonaccitheorem} is true, then it follows from a theorem of Dilworth~\cite{Dilworth50} that $Q(P_n)$ can also be partitioned into disjoint chains each of which intersects the largest layer(s) of $Q(P_n)$. Finding an explicit construction of such a partition appears difficult however: $Q(P_n)$ is asymmetric, and which layer is largest changes in an awkward and aperiodic way with $n$. It is fairly straightforward however to find such a partition for small $n$.

We begin with a partition of $Q(P_1)$ into a single chain $(\emptyset, \{1\})$, then build a partition for $Q(P_n)$ iteratively for $2 \leq n \leq 9$. 

Our chains shall come in three types: \emph{type A} chains are chains in $Q(P_n)$ every member of which contains $n$; \emph{type B} chains are chains in $Q(P_n)$ no member of which contains $n$; and \emph{type C} chains are chains in $Q(P_n)$ of length at least two where only the last member contains $n$. Our initial partition of $Q(P_1)$ thus consisted of a single C-chain.

Given such a partition of $Q(P_n)$, we build a partition of $Q(P_{n+1})$ into chains in the following way. 
\begin{itemize}
\item An \emph{A}-chain $(C_1\cup\{n\}, C_2\cup \{n\}, \ldots C_l\cup \{n\})$ in $Q(P_n)$ gives rise to a \emph{B}-chain in $Q(P_{n+1})$, namely $(C_1\cup\{n\}, C_2\cup \{n\}, \ldots C_l\cup \{n\})$.
\item A \emph{B}-chain $(C_1, C_2, \ldots C_l)$ in $Q(P_n)$ gives rise to (potentially) two chains in $Q(P_{n+1})$: a \emph{C}-chain $(C_1,C_2, \ldots C_l, C_l\cup\{n+1\})$, and (if $l>1$), to an \emph{A}-chain $(C_1\cup\{n+1\}, C_2\cup\{n+1\}, \ldots C_{l-1}\cup\{n+1\})$.
\item A \emph{C}-chain $(C_1, C_2, \ldots C_{l-1}, C_{l-1}\cup\{n\})$ in $Q(P_n)$ gives rise to two chains in $Q(P_{n+1})$: a \emph{B}-chain $(C_1, C_2,\  \ldots C_{l-1}, C_{l-1}\cup\{n\})$ and an \emph{A}-chain $(C_1\cup\{n+1\}, C_2\cup \{n+1\}, \ldots C_{l-2}\cup\{n+1\}, C_{l-1}\cup\{n+1\})$. (Note that by construction all \emph{C}-chains have length at least 2, so that each of them does indeed produce an \emph{A}-chain.)
\end{itemize}

It is easy to check that this iterative construction yields a partition of $Q(P_n)$ into chains through the largest layer for $n=1,2,\ldots 7$ and $n=9$. For $n=8$, we obtain a partition of $Q(P_8)$ containing one chain not intersecting the largest layer, $Q^{(3)}(P_8)$. However we can fix this by replacing the three chains $(\{258\})$, $(\{25\},\{257\})$ and $(\{57\})$ by the two chains $(\{25\},\{258\})$ and $(\{57\},\{257\})$. This establishes Conjecture~\ref{fibonacciconjecture} for all $n\leq 9$. The argument in the next subsection gives a simpler proof for $n=2,3 \ldots 7, 9$, and proves the additional case $n=10$.

\subsection{Shadows}
Another standard proof of Sperner's theorem (indeed Sperner's original proof) is to `push' an antichain towards the largest layer of $Q_n$ by repeatedly replacing the antichain's top-most layer by its lower shadow and the antichain's bottom-most layer by its upper shadow. Our proof of Theorem~\ref{fibonaccitheorem} is essentially a variant of this. Unfortunately, the out-degrees in $Q(P_n)$ are not sufficiently concentrated for this technique to give us even an approximate form of Conjecture~\ref{fibonacciconjecture}. We can however use shadow arguments to establish some small cases of Conjecture~\ref{fibonacciconjecture}.

For $n \geq 2$, set
\[ Q'(P_n) = \bigcup_{ \frac{n-1}{4}<r < \frac{n+2}{3}} Q^{(r)}(P_n).\] 
\begin{lemma}\label{shadow}
Let $n\geq 2$ and let $\mathcal{A}$ be an antichain in $Q(P_n)$. Then there exists an antichain $\mathcal{A}'$ in $Q'(P_n)$ with $\vert \mathcal{A}\vert \leq \vert \mathcal{A}' \vert$.
\end{lemma}
\begin{proof}
Let $\mathcal{A}$ be an antichain, and assume $\mathcal{A}$ is nonempty (for otherwise we have nothing to prove). Write $\mathcal{A}^{(r)}$ for the $r^{\textrm{th}}$ layer of  $\mathcal{A}$, 
\[\mathcal{A}^{(r)}= \mathcal{A} \cap Q^{(r)}(P_n).\]
Let $r_+(\mathcal{A})= \max\{r: \ \mathcal{A}^{(r)}\neq \emptyset\}$ and $r_-(\mathcal{A})= \min \{r: \ \mathcal{A}^{(r)}\neq \emptyset\}$. Suppose $r_+(\mathcal{A}) \geq \frac{n+2}{3}$. As $\mathcal{A}_0=\mathcal{A}$ is an antichain, we have that the family
\[\mathcal{A}_1= \left(\mathcal{A}\setminus \mathcal{A}^{(r_+)}\right)\cup \partial^-\mathcal{A}^{(r_+)}\]
is also an antichain. Now by counting edges between $\partial^-\mathcal{A}^{(r_+)}$ and $\mathcal{A}^{(r_+)}$ in the directed graph $D(P_n)$ we see that
\begin{align*}
\vert \partial^-\mathcal{A}^{(r_+)} \vert & \geq \frac{r_+}{n-2r_++2} \vert \mathcal{A}^{(r_+)} \vert \geq \vert \mathcal{A}^{(r_+)}\vert&&\textrm{(since $r_+\geq \frac{n+2}{3}$).}
\end{align*}
In particular $\vert \mathcal{A}_1 \vert \geq \vert \mathcal{A}_0 \vert$. Repeating this procedure as many times as necessary, we can produce an antichain at least as large as $\mathcal{A}$ with no set of size greater than or equal to $\frac{n+2}{3}$.

In the other direction, suppose $r_-(\mathcal{A}) \leq \frac{n-1}{4}$.  As $\mathcal{A}_0=\mathcal{A}$ is an antichain, we have that the family
\[\mathcal{A}_1= \left(\mathcal{A}\setminus \mathcal{A}^{(r_-)}\right)\cup \partial^+\mathcal{A}^{(r_-)}\]
is also an antichain. Counting edges between $\mathcal{A}^{(r_-)}$ and $\partial^+\mathcal{A}^{(r_-)}$  we have
\begin{align*}
\vert \partial^+\mathcal{A}^{(r_-)} \vert & \geq \frac{n-3r_-}{r_-+1} \vert \mathcal{A}^{(r_-)} \vert \geq \vert \mathcal{A}^{(r_-)}\vert&&\textrm{(since $r_-\leq \frac{n-1}{4}$).}
\end{align*}
In particular $\vert \mathcal{A}_1 \vert \geq \vert \mathcal{A}_0 \vert$. Repeating this procedure as many times as necessary, we can produce an antichain at least as large as $\mathcal{A}$ with no set of size less than or equal to $\frac{n-1}{4}$.

Now $\frac{n+2}{3}- \frac{n-1}{4}=\frac{n+11}{12}$, thus for $n \geq 2$ there always exists an integer $r:\ \frac{n-1}{4}< r <\frac{n+2}{3}$, so that the upper and lower shifting processes described above don't interfere with each other. So we can obtain from any antichain $\mathcal{A}$ an antichain $\mathcal{A}'$ which is at least as large and which lies in $Q'(P_n)$, as claimed.
\end{proof}

Observe now that for $n=2,3,4,5,6,7,9$ and $10$ there is a unique integer $r$ satisfying $ \frac{n-1}{4} < r < \frac{n+2}{3}$. Thus Conjecture~\ref{fibonacciconjecture} holds for these $n$. As we gave a partition of $Q(P_8)$ into chains meeting the largest layer in the previous subsection (and as the case $n=1$ is trivial), this means Conjecture~\ref{fibonacciconjecture} holds for all $n<11$.

By Lemma~\ref{shadow}, there is an antichain of maximum size in $Q(P_{11})$ which lies entirely inside $Q^{(3)}(P_{11})\cup Q^{(4)}(P_{11})$. The union of these two layers has size $154$, and the largest layer of $Q(P_{11})$ is $Q^{(3)}(P_{11})$ which has size $84$. Thus the first
open case of our conjecture asks whether we can find an antichain in $Q^{(3)}(P_{11})\cup Q^{(4)}(P_{11})$ with $85$ or more elements. This already does not look amenable to a pure brute force search.

\section{Theorem~\ref{fibonaccitheorem} for other graphs}\label{section: general G}

Our proof of Theorem~\ref{fibonaccitheorem} needed very little structural information about $Q(P_n)$. What we actually used was:
\begin{enumerate}[(i)]
\item the layer size $\vert Q^{(r)}(P_n)\vert$ increases monotonically with $r$ until it hits a maximum (or two consecutive maxima) and then decreases monotonically, and this maximum (or maxima) occurs (occur) when $r=r_{\star}=\alpha_{\star}n +O(1)$, where $\alpha_{\star}=\frac{5-5\sqrt{5}}{10}$;
\item for $\alpha=\alpha_{\star}+cn^{-\frac{1}{2}}$ and $c=o(\sqrt{n})$, we have $\vert Q^{(r)}(P_n)\vert\leq e^{-\gamma_1c^2+o(1)}\vert Q^{(r_{\star})}(P_n)\vert $, where $\gamma_1>0$ is a constant, and there are $o\left(\vert Q^{(r_{\star})}(P_n)\vert\right)$ sets in $Q(P_n)$ with size differing from $r_{\star}$ by more than $o(n)$;

\item within a layer, the number of sets with a given out-degree $\vert Q^{(r,d)}(P_n)\vert$ increases monotonically with $d$ until it hits a maximum (or two consecutive maxima) and then decreases monotonically.
For $r=\alpha n$ and $\alpha=\alpha_{\star}+o(1)$, this maximum (or maxima) occurs (occur) when $d=d_{\star}(r,n)=\beta_{\star}(\alpha)n +O(1)$, where $\beta_{\star}$ is a continuous function of $\alpha$;
\item for $\alpha=\alpha_{\star}+o(1)$, $\beta=\beta_{\star}(\alpha)+cn^{-\frac{1}{2}}$ and $c=o(\sqrt{n})$, we have $\vert Q^{(r,d)}(P_n)\vert\leq e^{-\gamma_2c^2+o(1)}\vert Q^{(r, d_{\star}(r))}(P_n)\vert $, where $\gamma_2>0$ is a constant, and there are $o\left(\frac{\vert Q^{(r)}(P_n)\vert}{\sqrt{n}}\right)$ sets in $Q^{(r)}(P_n)$ with out-degree differing from $d_{\star}(r,n)$ by more than $o(n)$;

\item for $r=r_{\star}+o(n)$ and $A \in Q^{(r)}(G_n)$, removing a vertex from $A$ increases its out-degree by at most $3$, and adding a vertex to $A$ decreases its out-degree by at most $3$.
\end{enumerate}
In fact, we could weaken (v): considering the case $t\geq 0$ only (the case $t\leq 0$ is similar) and re-using the notation from Section~\ref{section: shadows in trapezium}, it is sufficient for our argument that in each block $B_t$ there at most $O(q_n^{r-})$ `bad' sets $A$ from which we can remove a vertex and thereby increase the out-degree by more than $\gamma_3$, where $\gamma_3>0$ is a constant. Thus in turn it is enough if for each layer $Q^{(r)}(P_n)$ with $r=r_{\star}+o(n)$ there are at most $O\left(\frac{\vert Q^{(r)}(P_n)\vert}{\sqrt{n}}\right)$ `bad' sets $A$.

In particular, our proof of Theorem~\ref{fibonaccitheorem} actually gives the following more general result:
\begin{theorem}\label{general G theorem}
Let $\gamma_1, \gamma_2, \gamma_3>0$. Suppose $\left(G_n\right)_{n\in\mathbb{N}}$ is a sequence of $n$-vertex graphs satisfying the following properties:
\begin{enumerate}[(i)]
\item the layer size $\vert Q^{(r)}(G_n)\vert$ increases monotonically with $r$ until it hits a maximum (or two consecutive maxima) and then decreases monotonically, and this maximum (or maxima) occurs (occur) when $r=r_{\star}=\alpha_{\star}n +O(1)$, where $\alpha_{\star}\in(0,1)$ is a constant;
\item for $\alpha=\alpha_{\star}+cn^{-\frac{1}{2}}$ and $c=o(\sqrt{n})$, we have $\vert Q^{(r)}(G_n)\vert\leq e^{-\gamma_1c^2+o(1)}\vert Q^{(r_{\star})}(G_n)\vert $, and there are $o\left(\vert Q^{(r_{\star})}(G_n)\vert\right)$ sets in $Q(G_n)$ with size differing from $r_{\star}$ by more than $o(n)$;

\item within a layer, the number of sets with a given out-degree $\vert Q^{(r,d)}(G_n)\vert$ increases monotonically with $d$ until it hits a maximum (or two consecutive maxima) and then decreases monotonically.
For $r=\alpha n$ and $\alpha=\alpha_{\star}+o(1)$, this maximum (or maxima) occurs (occur) when $d=d_{\star}(r,n)=\beta_{\star}(\alpha)n +O(1)$, where $\beta_{\star}$ is a continuous function of $\alpha$;
\item for $\alpha=\alpha_{\star}+o(1)$, $\beta=\beta_{\star}(\alpha)+cn^{-\frac{1}{2}}$ and $c=o(\sqrt{n})$, we have $\vert Q^{(r,d)}(G_n)\vert\leq e^{-\gamma_2c^2+o(1)}\vert Q^{(r, d_{\star}(r))}(G_n)\vert $, and there are $o\left(\frac{\vert Q^{(r)}(G_n)\vert}{\sqrt{n}}\right)$ sets in $Q^{(r)}(G_n)$ with out-degree differing from $d_{\star}(r,n)$ by more than $o(n)$;

\item for $r=r_{\star}+o(n)$ there are at most $O\left(\frac{\vert Q^{(r)}(G_n)\vert}{\sqrt{n}}\right)$ sets $A \in Q^{(r)}(G_n)$ such that we can remove a vertex from $A$ and thereby increase its out-degree by more than $\gamma_3$, or add vertex to $A$ and thereby decrease its out-degree by more than $\gamma_3$.
\end{enumerate}
Then there exists a constant $\gamma_4>1$ such that 
\[s(G_n)\leq \gamma_4 \max_{0\leq r\leq n} \vert Q^{(r)}(G_n)\vert.\]
\end{theorem}
Theorem~\ref{general G theorem} covers for example the case when $G_n$ is the cycle $C_n$, or some finite power of $P_n$ or $C_n$. The calculations required to check that all the conditions above are satisfied in these cases are very similar to those we performed in Section~\ref{section: preliminaries}. For other graph families where the theorem might apply, the checks could however become more involved.

 We remark that the monotonicity condition in (i) is rather natural. Indeed, our example in Section~\ref{section: antichain in G-indep} of a graph sequence $G_n$ for which the width was of larger order than the size of a largest layer exploited precisely the non-monotonicity of the layer sizes.

 Similarly, (iii) and (iv) feel like reasonable conditions if we want to rule out antichain constructions spread over two consecutive layers and having size larger than the largest of the two layers by a factor of $1+\varepsilon$ for some $\varepsilon>0$ (e.g. by taking the union of the low out-degree sets in the bottom layer and the complement of their upper shadow).

 The requirement that $\alpha_{\star}\in (0,1)$ in condition (i) forces $G_n$ to have linear-sized independent sets. Given this and the monotonicity part of condition (i), the Chernoff-type concentration we require in condition (ii) is in fact what we would expect to see.

 Finally, (v) is a kind of homogeneity condition, chiming in with our intuition that a graph $G$ where `most' vertices look `more or less the same' should have width $s(G)$ `more or less the same' as the size of the largest layer in $Q(G)$.
 
 \begin{question}
 Suppose $(G_n)_{n\in \mathbb{N}}$ is a sequence of graphs satisfying all the conditions in Theorem~\ref{general G theorem}. Is it the case that 
 \[s(G_n)=(1+o(1))\max_{r} \vert Q^{(r)}(G_n)\vert?\]
 \end{question}

\section{Concluding remarks}\label{section: conclusion}
\subsection{The LYM inequality}
Sperner's theorem has over time given rise to an entire field, called Sperner Theory. We refer the reader to the monograph of Engel~\cite{Engel97} for more details on the subject. We have already briefly discussed two different proofs of Sperner's theorem in the previous section (via a partition into disjoint chains and via shadow arguments) and the reasons why they do not adapt well to the $Q(P_n)$ setting. Let us make a remark here about a third classical approach to Sperner's theorem, via the elegant LYM inequality of Bollob\'as, Lubell, Meshalkin and Yamamoto~\cite{Bollobas65, Lubell66, Meshalkin63, Yamamoto54}. 
\begin{theorem}[LYM inequality]
Let $n \in \mathbb{N}$ and $\mathcal{A} \subseteq Q_n$ be an antichain.Then
\begin{align*}
\sum_{r=0}^n \frac{ \vert \mathcal{A} \cap Q_n^{(r)}\vert}{\vert Q_n^{(r)}\vert} \leq 1.
\end{align*}
\end{theorem}
Note that Sperner's theorem is instant from LYM. Unfortunately we have been unable to find a good analogue of the LYM inequality for $Q(P_n)$. Not all maximal chains in $Q(P_n)$ have the same length, nor are elements in a given layer of $Q(P_n)$ contained in the same number of chains. Indeed, even restricting to `typical' layers and `typical' elements of those layers does not help us. As for shadows, the out-degrees are insufficiently concentrated for a uniform random chain to prove even an approximate form of Conjecture~\ref{fibonacciconjecture}: a divergence in the out-degree by an additive factor of $O(\sqrt{n})$ blows up to a divergence by a constant multiplicative factor in the number of chain-extensions of order $O(\sqrt{n})$. So to adapt the LYM strategy to our $Q(P_n)$ setting, we would need to construct a biased random chain which samples layers in a uniform manner. We could for example associate an `energy' to sets, which would be high on high out-degree sets, and then give our random chain a slight bias toward lower energy configurations. Though we have been unable to do this, it is probably one of the more promising approaches left open by our investigations.

\subsection{Isoperimetric questions}
One way we might try to construct a partition of $Q(P_n)$ into chains is to find for any pair of consecutive layers a matching in (the undirected version of) $D(P_n)$ from the smaller layer to the larger one. By Hall's marriage theorem~\cite{Hall48}, such matchings exist if and only if Hall's condition is satisfied in the bipartite subgraphs of $D(P_n)$ corresponding to consecutive layers of $Q(P_n)$ --- i.e. if and only if for every $r>r_{\star}$ and every $\mathcal{A} \subseteq Q^{(r)}(P_n)$ we have $\vert \mathcal{A} \vert \leq \vert \partial^- \mathcal{A} \vert$, and for every $r<r_{\star}$ and every $\mathcal{A} \subseteq Q^{(r)}(P_n)$ we have $\vert \mathcal{A} \vert \leq \vert \partial^+ \mathcal{A} \vert$.

This makes us interested more generally in the following isoperimetric problems.
\begin{problem}
Let $0 \leq r \leq \lceil \frac{n}{2}\rceil$ and let $0 \leq s \leq q_n^r$. Identify the families $\mathcal{A} \subseteq Q^{(r)}(P_n)$ of size $s$ that minimise the size of the lower shadow. 
\end{problem}  
\begin{problem}
Let $0 \leq r \leq \lceil \frac{n}{2}\rceil$ and let $0 \leq s \leq q_n^r$. Identify the families $\mathcal{A} \subseteq Q^{(r)}(P_n)$ of size $s$ that minimise the size of the upper shadow. 
\end{problem}
\begin{remark}
Since $Q(P_n)$ is not closed under complements, these two problems are not equivalent.
\end{remark}  
In the usual hypercube $Q_n$, these problems were solved by Kruskal and Katona~\cite{Katona68, Kruskal63} using shifting techniques that cannot be adapted to $Q(P_n)$ without additional ideas. Talbot~\cite{Talbot01} has moreover exhibited examples which show that the families minimising the size of the lower shadow in $Q(P_n)$ are not nested, suggesting the problem may be quite difficult.

\section*{Acknowledgements}
The author would like to thank David Saxton for many stimulating conversations on the problem, and the two anonymous referees for their careful work and helpful suggestions, which led to significant improvements in the presentation of this paper.

\bibliographystyle{plain}
\bibliography{masterbibliography}

\end{document}